\documentclass[12pt]{amsart}
\usepackage[dvips]{graphicx}

\newtheorem{thm}{Theorem}
\newtheorem{lem}[thm]{Lemma}

\theoremstyle{definition}

\theoremstyle{remark}

\numberwithin{equation}{section}

\usepackage{amsmath,amssymb,txfonts}
\usepackage{ulem}
\usepackage{wrapfig}
\usepackage{comment}
\usepackage{amsthm,color,bm}
\usepackage{cases}
\usepackage{comment}

\def\re{\mathbb{R}}
\def\N{\mathbb{N}}

\def\({\left(}
\def\){\right)}

\def\la{\lambda}

\begin{document}

\title[A simple proof of attainability for the Sobolev inequality]{A simple proof of attainability \\for the Sobolev inequality}

\author{Megumi Sano}
\address{Laboratory of Mathematics, School of Engineering,
Hiroshima University, Higashi-Hiroshima, 739-8527, Japan}
\email{smegumi@hiroshima-u.ac.jp}

\subjclass[2020]{Primary 47J30; Secondary 46B50.}

\keywords{Sobolev inequality, the best constant, cocompactness}
\date{\today}

\dedicatory{}

\begin{abstract}
We give a simple proof of the existence of a minimizer for the Sobolev inequality. Our proof is based on a representation formula via a cut-off fundamental solution. 
\end{abstract}

\maketitle

\section{Introduction}\label{S Intro}

Let $1<p<N, p^* = \frac{Np}{N-p}$ and $\dot{W}^{1,p}(\re^N)=\{ u \in L^{p^*} (\re^N) \mid  |\nabla u| \in L^p (\re^N) \}$. The Sobolev inequality states that
\begin{align*}
S_p \( \int_{\re^N} |u|^{p^*} \,dx \)^{\frac{p}{p^*}} \le \int_{\re^N} |\nabla u|^p \,dx \,\,\text{for all}\,u \in \dot{W}^{1,p}(\re^N),
\end{align*}
where $S_p$ is the Sobolev best constant which is given by
\begin{align*}
S_p &= \inf \left\{ \,\|\nabla u\|_p^p \,\,\Bigr| \,\, u \in \dot{W}^{1,p}(\re^N), \,\| u \|_{p^*} =1  \right\}\\
&= \inf \left\{ \,\|\nabla u\|_p^p \,\,\Bigr| \,\, u \in \dot{W}^{1,p}_{\rm rad}(\re^N), \,\| u \|_{p^*} =1  \right\},
\end{align*}
where we used the rearrangement argumet and $\dot{W}_{\rm rad}^{1,p}(\re^N)=\{\, u \in \dot{W}^{1,p} (\re^N) \mid u(x)= u(|x|) \,\,\text{is radial.} \}$.
In this short note, we give a simple proof of the following theorem.

\begin{thm}\label{Thm:Sob}
$S_p$ is attained.
\end{thm}

Theorem \ref{Thm:Sob} was firstly shown by Aubin \cite{Au} and Talenti \cite{T} based on rerrangement argument and the Hilbert invariant integral, see also \cite{Bliss, Al simple}. Other than that, Theorem \ref{Thm:Sob} was shown based on Concentration Compactness Principle via probability measures, see \cite{LionsI,Struwe}, based on blow-up analysis via a minimizer for the subcritical problem 
$$S(q)=\inf \left\{ \,\|\nabla u\|_2^2 \,\,\Bigr| \,\, u \in W_0^{1,2}(B_1), \,\| u \|_{q} =1  \right\}\,(q<2^*),$$
 see \cite{CC}, based on mass transportation approach via Brenier's map, see \cite{CNV} and based on Bellman's function method, see \cite{O}. 

\section{Cocompactness of the Sobolev embedding}\label{S emb}

It is well-known that the embedding $\dot{W}_{\rm rad}^{1,p} \subset L^{p^*}$ is not compact, that is, there exists a noncompact sequence $\{ u_k \}_{k \in \N} \subset \dot{W}_{\rm rad}^{1,p}(\re^N)$ such that $u_k \rightharpoonup 0$ in $\dot{W}_{\rm rad}^{1,p}(\re^N)$ as $k \to \infty$, but $u_k \not\to 0$ in $L^{p^*}(\re^N)$. Indeed, we can check it by using $u_k(|x|)= k^{\frac{N-p}{p}}u(k|x|)$ for smooth radial function $u$ with ${\rm supp}\, u \subset B_1 \subset \re^N$. 
Noncompactness makes analysis for variational problems like $S_p$ harder in general.  
However, the embedding $\dot{W}^{1,p}_{\rm rad} \subset L^{p^*}$ becomes {\it cocompact}. Cocompactness is weaker propoerty than compactness, but it is useful to show Theorem \ref{Thm:Sob}. For the details, see e.g. \cite{AOT, AT}. In this section, we give a simple proof of cocompactness of the embedding $\dot{W}^{1,p}_{\rm rad} \subset L^{p^*}$. Set the scaling $g_\la$ as follows.
\begin{align*}
g_{\la} u (x) 
= \la^{\frac{N-p}{p}} u (\la x) \quad (\la >0, \, x \in \re^N)
\end{align*}

\begin{lem}\label{Cocpt}(Cocompactness of the embedding $\dot{W}_{\rm rad}^{1,p} \subset L^{p^*}$)
Let $\{ u_k \}_{k \in \N} \subset \dot{W}_{\rm rad}^{1,p}(\re^N)$ satisfy $g_{\la_k} u_k \rightharpoonup 0$ in $\dot{W}_{\rm rad}^{1,p}(\re^N)$ for any $\la_k >0$. Then $u_k \to 0$ in $L^{p^*}(\re^N)$. 
\end{lem} 

\noindent
Following \cite{AOT, AT}, we introduce cut-off fundamental solution $m_t$ which is corresponding to Moser's sequence in the critical case $p=N$:
\begin{align*}
m_t (x)=m_t (|x|) = 
\begin{cases}
\( \frac{p-1}{N-p} \)^{\frac{p-1}{p}} | \mathbb{S}^{N-1}|^{-\frac{1}{p}} \, t^{-\frac{N-p}{p}} \quad &\text{if} \,\, |x| \in [0, t],\\
\( \frac{p-1}{N-p} \)^{\frac{p-1}{p}} | \mathbb{S}^{N-1}|^{-\frac{1}{p}} \,t^{\frac{N-p}{(p-1)p}} \,|x|^{-\frac{N-p}{p-1}}&\text{if} \,\, |x| \in (t, \infty)
\end{cases}
\end{align*}
We observe that $\| \nabla m_t \|_p =1$, for any $\la, t >0$ and $x \in \re^N$
\begin{align}\label{gm}
g_\la m_t (x)= m_{t/\la}(x)
\end{align}
and the following representation formula holds for a.e. $t>0$ and for all radial functions $w \in \dot{W}^{1,p}_{\rm rad} (\re^N)$. 
\begin{align}\label{rep}
\( \frac{N-p}{p-1} \)^{\frac{p-1}{p}} | \mathbb{S}^{N-1}|^{\frac{1}{p}} \,t^{\frac{N-p}{p}} w(t) = \int_{\re^N} |\nabla m_t (|x|) |^{p-2} \nabla m_t (|x|) \cdot \nabla w (|x|) \,dx
\end{align}
This formula is a key of the proof of Lemma \ref{Cocpt}.

\begin{proof}(Lemma \ref{Cocpt})
Let $g_{\la_k} u_k \rightharpoonup 0$ in $\dot{W}^{1,p}_{\rm rad}(\re^N)$ for any $\la_k>0$. Then we see that $g_{\la_k} u_k \rightharpoonup 0$ also in $L^{p^*} (\re^N)$. 
(\ref{gm}) and (\ref{rep}) imply 
\begin{align*}
o(1)  &= \int_{\re^N} |\nabla m_1 (|x|) |^{p-2} \nabla m_1 (|x|) \cdot \nabla g_{\la_k} u_k (|x|) \,dx\\
&= \int_{\re^N} |\nabla g_{1/\la_k} m_1 (|x|) |^{p-2} \nabla g_{1/\la_k} m_1 (|x|) \cdot \nabla u_k (|x|) \,dx\\
&= \int_{\re^N} |\nabla m_{\la_k} (|x|) |^{p-2} \nabla m_{\la_k} (|x|) \cdot \nabla u_k (|x|) \,dx\\
&=\( \frac{N-p}{p-1} \)^{\frac{p-1}{p}} | \mathbb{S}^{N-1}|^{\frac{1}{p}} \,\la_k^{\frac{N-p}{p}} u_k (\la_k)
\end{align*}
for any $\la_k >0$. Thus, we have
\begin{align*}
\underset{x \in \re^N}{\rm ess.sup}\, |x|^{\frac{N-p}{p}}|u_k (|x|)| \to 0.
\end{align*}
Since $\{ u_k \}_{k \in \N} \subset \dot{W}^{1,p}_{\rm rad}(\re^N)$ is the bounded sequence, 
\begin{align*}
\| u_k \|_{p^*}^{p^*}
&\le \( \underset{x \in \re^N}{\rm ess.sup}\, |x|^{\frac{N-p}{p}}|u_k (|x|)| \)^{p^* -p} \int_{\re^N} \frac{|u_k |^{p}}{|x|^p} \,dx\\
&\le \( \underset{x \in \re^N}{\rm ess.sup}\, |x|^{\frac{N-p}{p}}|u_k (|x|)| \)^{p^* -p} \( \frac{p}{N-p} \)^p \| \nabla u_k \|_p^p \to 0
\end{align*}
where in the last inequality we used the Hardy inequality. Finally, we get $u_k \to 0$ in $L^{p^*}(\re^N)$. 
\end{proof}

\section{Proof of Theorem \ref{Thm:Sob}}\label{S proof}

We use Brezis-Lieb lemma to show Theorem \ref{Thm:Sob}.

\begin{lem}\label{Lemma BL}(Brezis-Lieb \cite{BL lemma})
Let $p \in (0, +\infty)$ and $\{ g_m \}_{m=1}^{\infty} \subset L^p(\Omega)$ satisfy the followings.
\begin{itemize}
\item[(i)] There exists a constant $C$ such that $\| g_m \|_{L^p(\Omega, \mu)} \le C$ for all 
$m \in \mathbb{N}$.
\item[(ii)] $g_m(x) \to g(x)$ a.e. $x \in \Omega$ as $m \to \infty$.
\end{itemize}
Then 
\begin{equation*}
	\lim_{m \to \infty} ( \| g_m \|_{L^p(\Omega, \mu)}^p 
- \| g_m -g \|_{L^p(\Omega, \mu)}^p ) = \| g \|_{L^p(\Omega, \mu)}^p.
\end{equation*}
\end{lem}

\begin{proof}(Theorem \ref{Thm:Sob})
Let $\{ u_k \}_{k \in \N} \subset \dot{W}^{1,p}_{\rm rad}(\re^N)$ be a minimizing sequence of $S_p$. Then, up to a subsequence (we use the same notation for simplicity),
\begin{align*}
\| \nabla u_k \|_p^p \to S_p,\,\| u_k \|_{p^*} =1,\, u_k \rightharpoonup u \,\,{\rm in}\,\, \dot{W}^{1,p}_{\rm rad}(\re^N).
\end{align*}
We see that $u_k \rightharpoonup u$ also in $L^{p^*} (\re^N)$ and for any $R>0$
\begin{align*}
u_k|_{B_R} &\rightharpoonup u|_{B_R} \quad {\rm in}\,\,W^{1,p}(B_R) \,\,{\rm and}\,\,L^{p^*}(B_R),\\
u_k|_{B_R} &\to u|_{B_R} \quad {\rm in}\,\,L^q(B_R) \,\,\text{for}\,\,q < p^*,\\
u_k|_{B_R}(x) &\to u|_{B_R}(x) \quad \text{a.e.}\,\,x \in B_R.
\end{align*}
Since $R>0$ is arbitrary, we get
\begin{align*}
u_k(x) \to u(x) \quad \text{a.e.}\,\,x \in \re^N.
\end{align*}
Besides, we get
\begin{align}\label{a.e.}
\nabla u_k(x) \to \nabla u(x) \quad \text{a.e.}\,\, x \in \re^N.
\end{align}
We will show (\ref{a.e.}) later. 
Using Lemma \ref{Lemma BL}, we have
\begin{align*}
1=\| u_k \|_{p^*}^p
&= \( \| u_k -u\|_{p^*}^{p^*} + \| u \|_{p^*}^{p^*} \)^{\frac{p}{p^*}} + o(1)\\
&\le \( \| u_k -u\|_{p^*}^{p^*} \)^{\frac{p}{p^*}}  + \( \| u\|_{p^*}^{p^*} \)^{\frac{p}{p^*}} + o(1)\\
&\le S_p^{-1} \( \| \nabla (u_k -u)\|_{p}^{p} + \| \nabla u\|_{p}^{p} \)+ o(1)\\
&=S_p^{-1} \| \nabla u_k \|_{p}^{p} + o(1) =1+ o(1)
\end{align*}
which implies either
\begin{align*}
u\equiv 0\,(\text{\it Concentration}) \,\,\text{or}\,\,u_k \to u\,\,\text{in}\,\,L^{p^*}(\re^N)\,(\text{\it Compactness}),
\end{align*}
where we used the equality condition ($a=0$ or $b=0$) of the elementary inequality: $(a+b)^{q} \le a^q + b^q\,(a, b \ge 0, q\in (0,1))$. 

\noindent
[Case I: $u \not\equiv 0$] Then we get $1=\lim_{k \to \infty} \| u_k \|_{p^*} =\| u \|_{p^*}$. Therefore, 
$$S_p \le \|\nabla u \|^p_p \le \liminf_{k\to \infty} \|\nabla u_k \|^p_p = S_p$$
which implies that $u$ is a minimizer for $S_p$. 

\noindent
[Case II: $u \equiv 0$] Since $\| \nabla (g_{\la_k} u_k) \|_{p}=\| \nabla u_k \|_{p} < \infty$ for any $\la_k >0$, there exists $v \in \dot{W}_{\rm rad}^{1,p}(\re^N)$ such that, up to a subsequence, $g_{\la_k} u_k \rightharpoonup v$ in $\dot{W}_{\rm rad}^{1,p}(\re^N)$. If we assume that $g_{\la_k}u_k \rightharpoonup 0$ in $\dot{W}_{\rm rad}^{1,p}(\re^N)$ for any $\la_k >0$, Lemma \ref{Cocpt} implies that $u_k \to 0$ in $L^{p^*}(\re^N)$ which contradicts $1= \| u_k \|_{p^*}$. Therefore, there exists $\{ \la_k \}$ such that the new minimizing sequence $\{ g_{\la_k} u_k \}$ for $S_p$ satisfies $g_{\la_k} u_k \rightharpoonup v \not\equiv 0$ in $\dot{W}_{\rm rad}^{1,p}(\re^N)$. If we apply the same argument as [Case I], then we see that $v$ is a minimizer for $S_p$. 
\end{proof}

\noindent
If $p=2$, we do not need (\ref{a.e.}) because we get directly 
$$\| \nabla u_k \|_2^2 = \| \nabla (u_k -u) \|^2_2 + \| \nabla u \|_2^2 + o(1)$$
from $u_k\rightharpoonup u$ in $\dot{W}_{\rm rad}^{1,2}(\re^N)$. However, if $p \not= 2$, we need (\ref{a.e.}) to show it. Finally, we show (\ref{a.e.}).

\begin{proof}(Proof of (\ref{a.e.}))
Set
$$J(u)= \| \nabla u \|_p^p - S_p \| u \|_{p^*}^p \ge 0\quad (u \in \dot{W}_{\rm rad}^{1,p}(\re^N)).$$
Then 
$$0= \inf_{u \in \dot{W}_{\rm rad}^{1,p}(\re^N)} J(u) = J(u_k) + o(1)\,\,(k \to \infty).$$
Ekeland's variational principle \cite[Corollary 11]{E} implies that there exists $\{ v_k \} \subset \dot{W}_{\rm rad}^{1,p}(\re^N)$ such that 
\begin{itemize}
\item[(i)] $J(v_k) \le J(u_k)$ for any $k \in \N$,
\item[(ii)] $J'(v_k) \to 0$ in $\( \dot{W}_{\rm rad}^{1,p}(\re^N)  \)^*$ as $k \to \infty$,
\item[(iii)] $\| \nabla (u_k -v_k) \|_p \to 0$ as $k \to \infty$.
\end{itemize}
Since, up to a subsequence, $v_k \rightharpoonup u$ in $\dot{W}_{\rm rad}^{1,p}(\re^N)$ and $\| v_k \|_{p^*} = \| u_k \|_{p^*} + o(1) = 1 + o(1)$, $v_k$ satisfies 
\begin{align*}
o(1) &= \left| \frac{1}{p}\, J'(v_k)[\varphi] \right|\\
&= \left| \int_{\re^N} |\nabla v_k|^{p-2} \nabla v_k \cdot \nabla \varphi \,dx-S_p \int_{\re^N} |v_k |^{p^* -2}v_k \,\varphi \,dx \right| + o(1)
\end{align*}
for any $\varphi \in \dot{W}_{\rm rad}^{1,p}(\re^N)$. For $R, \eta>0$, we consider the test function 
$$\varphi = \phi_R T_{\eta} (v_k -u).$$
Here, $\phi_R$ is a smooth radial function such that $\phi_R \equiv 1$ on $B_R$ and $\phi_R \equiv 0$ on $\re^N \setminus B_{2R}$, and $T_\eta: \re \to \re$ is the truncation at height $\eta$ which is given by
\begin{align*}
T_\eta (s) = s \quad {\rm if}\,\, |s| \le \eta, \quad T_\eta (s) = \eta \,\frac{s}{|s|} \quad {\rm if}\,\, |s| > \eta.
\end{align*}
In the same way as \cite[Proof of Theorem 2.1]{BM}, we can show that $\nabla v_k (x) \to \nabla u(x)$ a.e. $x \in B_R$. Since $R >0$ is arbitrary, we get
\begin{align*}
\nabla v_k (x) \to \nabla u(x) \quad \text{a.e.}\,\, x \in \re^N.
\end{align*}
Since $\| \nabla (u_k -v_k) \|_p \to 0$, up to a subsequence, we have $\nabla (u_k-v_k)(x) \to 0$ a.e. $x \in \re^N$. Therefore, $\nabla u_k (x) \to \nabla u(x)$ a.e. $x \in \re^N$. 
\end{proof}



In a similar way, we can show the following theorem. We omit the proof. 

\begin{thm}\label{Thm:HS}(Hardy-Sobolev inequality)
Let $0 < s < p$ and $p^*(s) = \frac{N-s}{N-p}p$. Then
\begin{align*}
\inf \left\{ \,\|\nabla u\|_p^p \,\,\Bigr| \,\, u \in \dot{W}^{1,p}(\re^N), \int_{\re^N} \frac{|u |^{p^*(s)}}{|x|^{s}}\,dx =1  \right\}
\end{align*} 
is attained.
\end{thm}


\noindent
\textbf{Acknowledgment}: 
The author was supported by JSPS KAKENHI Early-Career Scientists, No. 23K13001. 
This work was partly supported by Osaka Central Advanced Mathematical Institute: MEXT Joint Usage/Research Center on Mathematics and Theoretical Physics JPMXP0619217849.

\end{document}